\documentclass[a4paper,leqno]{article}

\usepackage[T1]{fontenc}
\usepackage{ae,aecompl}
\usepackage{amsmath,amsthm,amssymb}
\usepackage[all]{xy}
\usepackage{enumitem}
\usepackage{aliascnt}
\usepackage{hyperref}
\usepackage{mathtools}

\usepackage{algorithm}
\usepackage[noend]{algorithmic}

\newcommand{\makedate}{\today}

\newtheorem{theorem}{Theorem}[section]

\newaliascnt{lemma}{theorem}
\newtheorem{lemma}[lemma]{Lemma}
\aliascntresetthe{lemma}

\newaliascnt{proposition}{theorem}
\newtheorem{proposition}[proposition]{Proposition}
\aliascntresetthe{proposition}

\newaliascnt{corollary}{theorem}
\newtheorem{corollary}[corollary]{Corollary}
\aliascntresetthe{corollary}

\newaliascnt{remark}{theorem}
\newtheorem{remark}[remark]{Remark}
\aliascntresetthe{remark}

\newaliascnt{definition}{theorem}

\aliascntresetthe{definition}

\newaliascnt{example}{theorem}
\newtheorem{example}[example]{Example}
\aliascntresetthe{example}

\newaliascnt{conjecture}{theorem}

\aliascntresetthe{conjecture}

\newaliascnt{openproblem}{theorem}

\aliascntresetthe{openproblem}

\newaliascnt{exercise}{theorem}

\aliascntresetthe{exercise}

\newaliascnt{algo}{theorem}

\aliascntresetthe{algo}

\newaliascnt{assumption}{theorem}

\aliascntresetthe{assumption}

\newaliascnt{notation}{theorem}
\newtheorem{notation}[notation]{Notation}
\aliascntresetthe{notation}

\newcounter{namedenum}
  {\end{list}}
\newenvironment{myexample}{\begin{example}\upshape}{\end{example}}
\newenvironment{myremark}{\begin{remark}\upshape}{\end{remark}}

\newcommand{\Magma}{\textsc{Magma}}
\newcommand{\A}{\mathcal{A}}
\newcommand{\NN}{\mathbb{N}}
\newcommand{\ZZ}{\mathbb{Z}}
\DeclareMathOperator{\gr}{\mathrm{gr}}
\DeclareMathOperator{\sgn}{\mathrm{sgn}}
\DeclareMathOperator{\rwedge}{\widetilde{\wedge}}
\DeclareMathOperator{\rvee}{\widetilde{\vee}}
\newcommand{\PP}{\mathcal{P}}

\DeclareMathOperator{\rbigvee}{\widetilde{\bigvee}}
\newcommand{\Sym}{\mathfrak{S}}
\DeclareMathOperator{\Orb}{\mathrm{Orb}}
\newcommand{\T}{\mathfrak{T}}
\newcommand{\Part}{\mathfrak{P}}
\newcommand{\TT}{\widetilde{T}}

\begin{document}

\title{Computing growth functions of braid monoids and counting vertex-labelled bipartite graphs}
\author{Volker Gebhardt\footnotemark[2]}
\date{\makedate}

\renewcommand{\thefootnote}{\fnsymbol{footnote}}
\footnotetext[2]{The author acknowledges support under Australian Research Council's Discovery Projects funding scheme (project number DP1094072), and the Spanish Project MTM2010-19355.}
\maketitle
\renewcommand{\thefootnote}{\arabic{footnote}}


\begin{abstract}
\noindent
We derive a recurrence relation for the number of simple vertex-labelled bipartite graphs with given degrees of the vertices and use this result to obtain a new method for computing the growth function of the Artin monoid of type $A_{n-1}$ with respect to the simple elements (permutation braids) as generators.
Instead of matrices of size $2^{n-1}\times 2^{n-1}$, we use matrices of size $p(n)\times p(n)$, where $p(n)$ is the number of partitions of $n$.
\end{abstract}


\section{Introduction}\label{S:Intro}%
Let $M$ be a monoid with a finite generating set $S$, let $S^*$ be the free monoid over $S$, and let $\alpha:S^*\to M$ be the natural epimorphism.

For $w\in S^*$ we denote by $|w|$ the length of $w$, and for $g\in M$ we denote by $|g|_S$ the length of a shortest word $w\in S^*$ satisfying $\alpha(w)=g$.
We call $w\in S^*$ a \emph{geodesic representative with respect to $S$} for $g\in M$, if $\alpha(w)=g$ and $|w|=|g|_S$.

The \emph{spherical growth function} of $M$ with respect to $S$ is defined as the formal power series
\[
  \gr(t) = \gr_{M,S}(t) = \sum_{g\in M}t^{|g|_S} = \sum_{\ell=0}^\infty a_\ell t^\ell \;,
\]
where $a_\ell$ is the number of elements of $M$ that have a geodesic representative of length $\ell$ with respect to $S$.  That is, computing the spherical growth function is equivalent to counting the elements of given geodesic length.

It is well known that the spherical growth function is a rational function, if $M$ admits a \emph{geodesic automatic structure} with respect to $S$~\cite{Epstein1996}; the latter in particular is the case, if each element of $M$ admits a unique \emph{normal form} which is a geodesic representative, and if there exists a finite state automaton recognising normal forms.
Geodesic automatic structures of this kind have been described for the braid monoid~\cite{Epstein1992} and, more generally, Artin monoids of finite type~\cite{Charney1993} and (yet more generally) Garside monoids of spherical type~\cite{Dehornoy2002}.

The \emph{transition matrix} of a finite state automaton with states $\{1,\ldots,N\}$ is an $N\times N$ matrix $T\in\ZZ^{N\times N}$ whose entry $T_{i,j}$ is the number of transitions leading from state $j$ to state $i$.
If $T$ is the transition matrix of the finite state automaton recognising normal forms, with state $1$ being the starting state, $u=(1,0,\ldots,0)^t$, and if $v\in\ZZ^{1\times N}$ is the characteristic function of the set of accept states, then the number of accepted strings of length $\ell$, that is the number of elements of geodesic length $\ell$, can be computed as
$a_\ell = v T^\ell u$~\cite{Epstein1996}
\smallskip

For Artin monoids, two finite state automata recognising normal forms have been described in~\cite{Charney1993}; cf.\ \autoref{S:Background}.
The states of the first automaton are in bijection to the \emph{simple elements} of the monoid, which in the case of Artin monoids correspond to the elements of the associated Coxeter group;
this automaton readily generalises to Garside monoids of spherical type.
The states of the second automaton are in bijection to the \emph{finishing sets} of simple elements;
the finishing set of a simple element is the set of atoms of $M$ that can occur last in any reduced word representing the given simple element.

In the important case of the Artin monoid of type $A_{n-1}$, the braid monoid $B_n^+$ on $n$ strands, these automata have $(n-1)!$ respectively $2^{n-1}$ states.
Due to their size, computing the growth function of the braid monoid $B_n^+$ using the methods described in~\cite{Charney1993} becomes infeasible in practice, if $n\gtrsim10$.
\smallskip

Being able to compute the growth function, respectively the number of elements of given geodesic length, is of practical importance, among other things, for the purpose of generating uniformly random elements.
If $a_\ell$ is known, and if the elements of geodesic length $\ell$ can be enumerated effectively, then a uniformly random element of geodesic length $\ell$ can be chosen.
Such a uniform random generator has recently been described for the braid monoid with the set of atoms as generating set~\cite{GGM2012}.  (That is, a positive braid of given \emph{word length} is chosen with a uniform probability.)  For many purposes, however, being able to generate uniformly random braids with given length with respect to the set of simple braids would be more useful.  This problem is significantly harder; improving the methods for computing growth functions is a crucial prerequisite.%
\smallskip

In this paper, we give a new method to compute the growth function of the braid monoid $B_n^+$ that uses a matrix of size $p(n)\times p(n)$, where $p(n)$ is the number of partitions of $n$.
The basic idea is as follows:  We start with the automaton from~\cite{Charney1993} whose states are the finishing sets of simple elements; these finishing sets can be identified with the standard parabolic subgroups of the braid group $B_n$, respectively those of the associated Coxeter group $\Sym_n$.  We will show that it is possible to combine states for which the sizes of the orbits of certain subgroups of $\Sym_n$ agree, arriving at a reduced transition matrix whose rows and columns are indexed by the partitions of $n$.
The entries of the latter can be obtained from the numbers of simple vertex-labelled bipartite graphs with $n$ edges with given degree functions; a vertex of degree $k$ representing an orbit of size $k$.

The structure of the paper is as follows:
In \autoref{S:Counting} we derive a recursive formula for counting simple vertex-labelled bipartite graphs with $n$ edges.
In \autoref{S:Background}, we recall the Artin Garside structure of the braid monoid $B_n^+$, simple elements, greedy normal forms, as well as starting sets and finishing sets and the finite state automata recognising normal forms described in~\cite{Charney1993}; experts may skip this section.
In \autoref{S:GrowthFunction}, we establish that computing the growth function of the braid monoid $B_n^+$ can be reduced to counting simple vertex-labelled bipartite graphs with $n$ edges.
In \autoref{S:Feasibility}, finally, we discuss the feasibility of the new method.

After the journal version of this paper was accepted, we learnt that some the results obtained here may be connected to results from~\cite{Dehornoy2007}.


\section{Counting vertex-labelled bipartite graphs}\label{S:Counting}%
For a vertex-labelled bipartite graph $\Gamma=(U,V,E)$, where $U=\{u_1,\ldots,u_r\}$ and $V=\{v_1,\ldots,v_s\}$, we can consider the sequences $\deg_U=\big[\deg(u_1),\ldots,\deg(u_r)\big]$ and $\deg_V=\big[\deg(v_1),\ldots,\deg(v_s)\big]$ which satisfy $\deg(u_1)+\cdots+\deg(u_r)=|E|$ and $\deg(v_1)+\cdots+\deg(v_s)=|E|$.

We want to count the simple graphs yielding two given sequences. For $n\in\NN$ and sequences $A=[a_1,\ldots,a_r]$ and $B=[b_1,\ldots,b_s]$ of non-negative integers with
$\sum A = a_1+\cdots+a_r = n = \sum B = b_1+\cdots +b_s$, let $N_{A,B}$ be the number of simple vertex-labelled bipartite graphs $(U,V,E)$ with $\deg_U=A$ and $\deg_V=B$.

\begin{lemma}\label{L:CntGraphs-Symm}%
Let $A=[a_1,\ldots,a_r]$ and $B=[b_1,\ldots,b_s]$, where $\sum A=\sum B$.
\begin{enumerate}[label=\upshape{(\roman{*})},ref=\roman{*}]\vspace{-\itemsep}
\item\label{L:CntGraphs-Symm-swap}
$N_{A,B} = N_{B,A}$.
\item\label{L:CntGraphs-Symm-permut}
If $A'$ and $B'$ are equal to $A$ respectively $B$ up to permutation of entries and addition or removal of entries that are equal to $0$, then $N_{A',B'} = N_{A,B}$.
\end{enumerate}
\end{lemma}
\begin{proof}
Trivial.
\end{proof}

\begin{lemma}\label{L:CntGraphs-Basis}%
If $A=[a_1,\ldots,a_r]$ and $B=[1,\ldots,1]$, where $\sum A=n=\sum B$, then
\[
N_{A,B} = \frac{n!}{a_1!\cdots a_r!} \;.
\]
\end{lemma}
\begin{proof}
The graphs $(U,V,E)$ with $\deg_U=A$ and $\deg_V=B=[1,\ldots,1]$ are in bijection to the functions $\{1,\ldots,n\}\to\{1,\ldots,r\}$ for which the image $i$ occurs exactly $a_i$ times.
Thus, $N_{A,B}$ is the number of permutations of a multiset whose sequence of multiplicities is~$A$.
\end{proof}

\begin{proposition}\label{P:CntGraphs-Recur}%
If $A=[a_1,\ldots,a_r]$ and $B=[b_1,\ldots,b_s]$, where $\sum A=\sum B$, then
\[
N_{A,B} = \sum_{\substack{I\subseteq\{1,\ldots,r\}\\[0.5ex] |I|=b_1}} N_{A_I,B'} \;,
\]
where $B'=[b_2,\ldots,b_s]$ and $A_I=[a_1',\ldots,a_r']$ with
$a_i' = a_i - \mathbf{1}_I(i)$.
(\/$\mathbf{1}_I$ is the indicator function of the set $I$.)
\end{proposition}
\begin{proof}
Consider the vertex $v_1$ of degree $b_1$.  As we want to count simple graphs, there can be at most one edge from $v_1$ to every vertex in $U$; the $b_1$ edges ending at $v_1$ thus determine a unique $b_1$-subset of $\{1,\ldots,r\}$, and vice versa.

For a fixed configuration of edges ending at $v_1$, the remaining $n-b_1$ edges form a simple vertex-labelled bipartite graphs $(U',V',E')$, for which $\deg_{U'}=A_I$ and $\deg_{V'}=B'$ with $A_I$ and $B'$ as in the statement of the proposition.
Thus the claim is shown.
\end{proof}

By \autoref{L:CntGraphs-Symm}, $N_{A,B}$ only depends on the types of the partitions of $n$ given by $A$ respectively $B$, and we have $N_{A,B}=N_{B,A}$.  There are thus $\tfrac12p(n)\big(p(n)+1\big)$ numbers to compute, where $p(n)$ is the number of partitions of $n$.

\begin{corollary}\label{C:CntGraphs-Compl}%

The numbers $N_{A,B}$ for $\sum A=\sum B\le n$ can be computed in time $O\big(c^ {\sqrt{n}}\big)$ with $c<569$.
\end{corollary}
\begin{proof}
First assume that $N_{A',B'}$ is known and bounded by $f(n-1)$ for all sequences $A'$ and $B'$ with $\sum A'=\sum B'\le n-1$.

For $S=[s_1,\ldots,s_t]$ define $\min(S) = \min\{s_1,\ldots,s_t\}$ and $|S|=t$.
By \autoref{L:CntGraphs-Symm}, we may assume that $b_1=\min(B)\le\min(A)$ in the statement of \autoref{P:CntGraphs-Recur}, and thus $|A|\le\frac{n}{b_1}$.
The equality for $N_{A,B}$ therefore involves
\[
 \binom{|A|}{b_1}
 \le \binom{\lfloor\frac{n}{b_1}\rfloor}{b_1}
 \le \left(\frac{ne}{b_1^2}\right)^{b_1}
 \le e^{2\sqrt{\frac{n}{e}}}
\]
terms, whence $N_{A,B}$ can be computed in time $O\Big(e^{2\sqrt{\frac{n}{e}}}\,\ln f(n-1)\Big)$.
Moreover, $\ln f(n)\le\ln f(n-1)+2\sqrt{\frac{n}{e}}$.
As $f(1)=1$ by \autoref{L:CntGraphs-Basis}, one has $\ln f(n)\le n^2$.

Thus, knowing $N_{A',B'}$ for all sequences $A'$ and $B'$ with $\sum A'=\sum B'\le n-1$, the numbers $N_{A,B}$ for all sequences $A$ and $B$ with $\sum A=\sum B=n$ can be computed in time $O\Big(|\Part|^2n^2e^{2\sqrt{\frac{n}{e}}}\Big)=O\big(\widetilde{c}^ {\sqrt{n}}\big)$ with $\widetilde{c}<569$, where the last step follows with the asymptotic expression~\cite{HardyRamanujan1918}
\[
|\Part|=p(n)\sim\frac1{4n\sqrt3} e^{\pi\sqrt{\frac{2n}{3}}} \;.
\]
As $\sum_{m=1}^n \widetilde{c}^ {\sqrt{m}} < n\widetilde{c}^{\sqrt{n}}$ and $n\in O\big(a^{\sqrt{n}}\big)$ for any $a>1$, the claim follows.
\end{proof}

\begin{myremark}\label{R:CntGraphs}%
The choice $b_1=\min(B)\le\min(A)$ which was used in the proof of \autoref{C:CntGraphs-Compl} need not yield the smallest possible number of terms.  In practice, one should compare the binomial coefficients
\[
\binom{|A|}{\min(B)} \;,\;\;\;
\binom{|A|}{\max(B)} \;,\;\;\;
\binom{|B|}{\min(A)} \;\;\text{and}\;\;\;
\binom{|B|}{\max(A)} \;,
\]
where $\max\big([s_1,\ldots,s_t]\big) = \max\{s_1,\ldots,s_t\}$, and apply \autoref{P:CntGraphs-Recur} to the situation that corresponds to the smallest binomial coefficient.
\end{myremark}


\section{Garside structure and geodesic automation}\label{S:Background}%
In this section, we recall the finite state automata recognising normal forms that were described in~\cite{Charney1993}, using the terminology of Garside monoids~\cite{Dehornoy2002}.
For details, we refer to~\cite{Bessis2003,BrieskornSaito1972,Charney1993,Dehornoy2002,Deligne1972,El-RifaiMorton1994,Epstein1992}.
Experts may skip this section.

\subsection{Garside normal forms}\label{S:NormalForms}%
Garside monoids of spherical type were shown to be automatic in~\cite{Dehornoy2002}.
The automatic structure of a general Garside monoid of spherical type is closely related to the \emph{greedy normal form}; it is a generalisation of one of the automatic structures described in~\cite{Charney1993}.
\medskip

In a cancellative monoid $M$, we can define the \emph{prefix} partial order:
For $x,y\in M$, we say $x\preccurlyeq y$ if there exists $c\in M$ such that $xc=y$.
Similarly, we define the \emph{suffix} partial order by saying that $x\succcurlyeq y$ if there exists $c\in M$ such that $x=cy$.
We call $s\in M$ an \emph{atom}, if $s=ab$ (with $a,b\in M$) implies $a=1$ or $b=1$.

A cancellative monoid $M$ is called a \emph{Garside monoid of spherical type}, if it is a lattice (i.e., least common multiples and greatest common divisors exist and are unique) with respect to $\preccurlyeq$ and with respect to $\succcurlyeq$, if there are no strict infinite descending chains with respect to either $\preccurlyeq$ or $\succcurlyeq$, and if there exists an element $\Delta\in M$, such that
$D=\{s\in M \mid s\preccurlyeq\Delta\}=\{s\in M \mid \Delta\succcurlyeq s\}$ is finite and generates $M$.
In this case, we call $\Delta$ a \emph{Garside element} and the elements of $D$ the \emph{simple elements} (with respect to $\Delta$).  Moreover, we denote the $\preccurlyeq$-gcd and the $\preccurlyeq$-lcm of $x,y\in M$ by $x\wedge y$ respectively $x\vee y$, and the $\succcurlyeq$-gcd and the $\succcurlyeq$-lcm of $x,y\in M$ by $x\rwedge y$ respectively $x\rvee y$.
It follows that, for $s\in D$, there exists a unique element $\partial s\in D$ such that $s\,\partial s=\Delta$.
\medskip

We assume for the rest of this section that $M$ is a Garside monoid as above.
As $D$ generates $M$, every element $x\in M$ can be written in the form $x=x_1\cdots x_k$ with $x_1,\ldots,x_k\in D$.
The representation as a product of this form can be made unique by requiring that each simple factor is non-trivial and maximal with respect to $\preccurlyeq$.  More precisely, we say that $x=x_1\cdots x_k$ is in \emph{(left) normal form}, if $x_k\ne 1$ and if $x_i = \Delta\wedge(x_i\cdots x_k)$ for $i=1,\ldots,k$.
Equivalently, we can require
\begin{equation}\label{E:NormalForm}
 x_k\ne 1 \quad\text{and}\quad \partial x_i \wedge x_{i+1} = 1 \text{\, for \,} i=1,\ldots,k-1 \;.
\end{equation}
Since $M$ is cancellative, it follows easily by induction that the normal form of $x\in M$ is a geodesic representative of $x$ with respect to the generating set $D$.

Importantly, (\ref*{E:NormalForm}) is a \emph{local} condition, only involving pairs of adjacent factors.  Moreover, if $x_1\cdots x_k\in D^*$ is in normal form, then so is every initial subword $x_1\cdots x_\ell$ for $\ell\le k$.
Thus, we can check whether a word $x_1\cdots x_k\in D^*$ is in normal
form by reading it from the left to the right, only keeping track of the last read symbol:\,
If $x_1\cdots x_\ell$ is in normal form, testing whether $x_{\ell+1}\ne1$ and $\partial x_\ell \wedge x_{\ell+1}=1$ suffices to decide whether $x_1\cdots x_{\ell+1}$ is in normal form.

In other words, the following finite state automaton recognises normal forms:
\begin{center}\begin{tabular}{llll}
\hline
 \textbf{states}\rule[10pt]{0pt}{0pt}
  & \textbf{initial state}
  & \textbf{accept states}
  & \textbf{transition function}
\\ \hline
 $D$\rule[-15pt]{0pt}{37pt}
  & $\Delta$
  & $D\setminus\{1\}$
  & $(X,x_i)\mapsto\begin{cases}
                      x_i & \text{if $\partial X\wedge x_i=1$}\\
                      1   & \text{if $\partial X\wedge x_i\ne1$}
                   \end{cases}$ \\ \hline
\end{tabular}\end{center}
Note that, since $\partial\Delta=1$, the automaton will be in an accept state after reading the first symbol $x_1$, unless $x_1=1$.  (A single non-trivial simple element is in normal form.)
Note further that, since $\partial1=\Delta$, once the automaton has reached the fail state~$1$, it will remain in this state.
\medskip

The Garside structure of Artin monoids was described in~\cite{BrieskornSaito1972,Deligne1972}.
If $M$ is an Artin monoid, its simple elements are lifts of the elements of the corresponding Coxeter group $W$, its Garside element is the lift of the longest element of $W$, and its atoms are lifts of the simple reflections in $W$.%
\footnote{We remark that Artin \emph{groups}, apart from the Garside structure used here, admit another Garside structure, the so-called \emph{dual Garside structure}~\cite{Bessis2003}; the Garside structure used here is known as the \emph{classical Garside structure}.  However, the monoids (of positive elements) associated to these two Garside structures of an Artin group are different.  The monoid we refer to as ``Artin monoid of type $T$'' is the monoid of positive elements of the Artin group of type $T$ with respect to the Coxeter-type generators.}
Moreover, the Garside element is equal to both the $\preccurlyeq$-lcm of the atoms and the $\succcurlyeq$-lcm of the atoms.

In the case of the Artin monoid of type $A_{n-1}$, the braid monoid $B_n^+$ on $n$ strands, the set of simple elements is in bijection to the symmetric group $\Sym_{n-1}$, so the automaton arising from the Garside structure has $(n-1)!$ states.

\subsection{Starting and finishing sets}\label{S:FinishingSets}%
In the special case of Artin monoids, an element $x\in M$ is simple if and only if it is square-free, that is, if and only if it cannot be written as $x=ua^2v$ with $a,u,v\in M$ and $a\ne 1$.
This property makes it possible to recognise normal forms with an automaton involving fewer states.
We assume for the rest of this section that $M$ is an Artin monoid with Garside element $\Delta$ and set of simple elements $D$.  Let $\A$ denote the set of atoms of $M$.
\medskip

Given $x\in M$, we define the \emph{starting set} of $x$ as $S(x)=\{a\in \A \mid a\preccurlyeq x\}$ and the \emph{finishing set} of $x$ as $F(x)=\{a\in \A \mid x\succcurlyeq a\}$.
For any $x\in D$, we have $F(x)\cap S(\partial x)=\emptyset$ and $F(x)\cup S(\partial x)=\A$~\cite[Lemma 4.2]{Charney1993}.  For $u,v\in D$, one therefore has $\partial u\wedge v=1$ if and only if $S(v)\subseteq \A\setminus S(\partial u)= F(u)$.

We can thus check whether a word $x_1\cdots x_k\in D^*$ is in normal
form by reading it from the left to the right, only keeping track of the finishing set of the last read symbol:\,  If $x_1\cdots x_\ell$ is in normal form, testing whether $x_{\ell+1}\ne1$ and $S(x_{\ell+1})\subseteq F(x_{\ell})$ suffices to decide whether $x_1\cdots x_{\ell+1}$ is in normal form.

In other words, the following finite state automaton recognises normal forms; $\PP(\A)$ denotes the power set of the set $\A$ of atoms:
\begin{center}\begin{tabular}{llll}
\hline
 \textbf{states}\rule[10pt]{0pt}{0pt}
  & \textbf{initial state}
  & \textbf{accept states}
  & \textbf{transition function}
\\ \hline
 $\PP(\A)$\rule[-15pt]{0pt}{37pt}
  & $\A$
  & $\PP(\A)\setminus\{\emptyset\}$
  & $(X,x_i)\mapsto\begin{cases}
                      F(x_i)    & \text{if $S(x_i)\subseteq X$}\\
                      \emptyset & \text{if $S(x_i)\not\subseteq X$}
                   \end{cases}$ \\ \hline
\end{tabular}\end{center}
Note that $F(x_i)=\emptyset$ if and only if $x_i=1$.
Thus, the automaton will be in an accept state after reading the first symbol~$x_1$, unless $x_1=1$.
Moreover, once the automaton has reached the fail state~$\emptyset$, it will remain in this state.
\medskip

In the case of the Artin monoid of type $A_{n-1}$, the braid monoid $B_n^+$ on $n$ strands, the set of atoms has $n-1$ elements, so the automaton described in this section has $2^{n-1}$ states.

\begin{myexample}\label{E:B4}
Consider the braid monoid
\[
B_4^+=\big\langle\sigma_1,\sigma_2,\sigma_3
\mid
\sigma_1\sigma_2\sigma_1=\sigma_2\sigma_1\sigma_2,\;
\sigma_2\sigma_3\sigma_2=\sigma_3\sigma_2\sigma_3,\;
\sigma_1\sigma_3=\sigma_3\sigma_1\big\rangle\;.
\]
The set of atoms is $\A=\{ \sigma_1,\sigma_2,\sigma_3\}$, and the simple elements are in bijection to~$\Sym_4$.
Thus, the automaton has 8 states, and there are 24 transitions originating from any given state.  The states and the transition matrix are given below.

Observe that the rank of $T$ is $5$.
Columns $1$, $3$ and $8$ of $T$ are unique, but the following columns of $T$ are identical:  columns $2$ and $5$; columns $4$, $6$, and $7$.

\begin{center}
\begin{tabular}{cc}
\textbf{state} & \textbf{finishing set} \\ \hline
\rule[11pt]{0pt}{0pt}%
 $1$ & $\emptyset$ \\[1pt]
 $2$ & $\{ \sigma_1 \}$ \\[1pt]
 $3$ & $\{ \sigma_2 \}$ \\[1pt]
 $4$ & $\{ \sigma_1, \sigma_2 \}$ \\[1pt]
 $5$ & $\{ \sigma_3 \}$ \\[1pt]
 $6$ & $\{ \sigma_1, \sigma_3 \}$ \\[1pt]
 $7$ & $\{ \sigma_2, \sigma_3 \}$ \\[1pt]
 $8$ & $\A$
\end{tabular}
\qquad
\begin{minipage}[b]{7cm}
\[
T=\begin{pmatrix*}[r]
24 & 21 & 19 & 13 & 21 & 13 & 13 & \;1 \\
 0 &  1 &  1 &  2 &  1 &  2 &  2 &   3 \\
 0 &  1 &  2 &  3 &  1 &  3 &  3 &   5 \\
 0 &  0 &  0 &  1 &  0 &  1 &  1 &   3 \\
 0 &  1 &  1 &  2 &  1 &  2 &  2 &   3 \\
 0 &  0 &  1 &  2 &  0 &  2 &  2 &   5 \\
 0 &  0 &  0 &  1 &  0 &  1 &  1 &   3 \\
 0 &  0 &  0 &  0 &  0 &  0 &  0 &   1
\end{pmatrix*}
\]
\end{minipage}
\end{center}
\end{myexample}


\section{Growth function of the braid monoid\texorpdfstring{ $B_n^+$}{}}\label{S:GrowthFunction}%
For an Artin monoid with set of atoms $\A$ we define $\bigvee X = x_1\vee\cdots \vee x_k\in D$ for $\emptyset\neq X=\{x_1,\ldots,x_k\}\subseteq\A$ and $\bigvee\emptyset=1\in D$.
We analogously define $\rbigvee X$.
\begin{lemma}\label{L:Transitions}
Let $M$ be an Artin monoid with set of atoms $\A$ and set of simple elements $D$, and let $X,Y\subseteq\A$.
Then
\begin{multline*}
\lefteqn{\Big|\big\{ s\in D\mid S(s)\subseteq X \text{\, and \;} Y\subseteq F(s) \big\}\Big|} \\
= \bigg|\left\{ t\in D\;\;\big|\;\;
                     t\big({\textstyle\rbigvee} Y\big)\wedge\big({\textstyle\bigvee}\A\setminus X\big)=1
    \text{\, and \;} t\rwedge\big({\textstyle\rbigvee} Y\big)=1 \right\}\bigg|
\end{multline*}
\end{lemma}
\begin{proof}
The submonoid of $M$ generated by $Y$ is an Artin monoid with Garside element $\bigvee Y$~\cite{BrieskornSaito1972}.  In particular, $\rbigvee Y=\bigvee Y$.
As an element of $M$ is simple if and only it is square-free, this implies that $t\big(\rbigvee Y\big)$, for $t\in D$, is simple if and only if no atom in $Y$ is a suffix of $t$, that is, if and only if $t\rwedge\big(\rbigvee Y\big)=1$.

For $s\in D$ one has $S(s)\subseteq X$ if and only if $s\wedge\left(\bigvee\A\setminus X\right)=1$, and
$Y\subseteq F(s)$ if and only if $s\succcurlyeq\rbigvee Y$.
Moreover, $s\succcurlyeq\rbigvee Y$ with $s\in D$ is equivalent to $s=t\big(\rbigvee Y\big)$ for $t\in D$ such that $t\big(\rbigvee Y\big)$ is simple, that is, such that $t\rwedge\big(\rbigvee Y\big)=1$
\end{proof}

For the remainder of the paper, we restrict our attention to the braid monoid~$B_n^+$.  The set of atoms is $\A=\{\sigma_1,\ldots,\sigma_{n-1}\}$ and the canonical map $\pi:D\to\Sym_n$ is a bijection, where $\pi(\sigma_i)$ is the transposition $(i\hspace{8pt}i+1)$.
In particular, $s\in D$ is uniquely determined by the permutation $\pi_s=\pi(s)\in \Sym_n$ it induces on the $n$ strands.  More precisely, for $s\in D$ and $i=1,\ldots,n-1$, one has $\sigma_i\preccurlyeq s$ if and only if $\pi_s(i)>\pi_s(i+1)$, and $s\succcurlyeq\sigma_i$ if and only if $\pi_s^{-1}(i)>\pi_s^{-1}(i+1)$~\cite{Epstein1992}.

For any $S\subseteq\A$, the subgroup $\langle\pi(S)\rangle$ of~$\Sym_n$ generated by $\pi(S)$ is a standard parabolic subgroup of~$\Sym_n$.
Each orbit of $\langle\pi(S)\rangle$ is a set of consecutive points; an orbit of $\langle\pi(S)\rangle$ of size $s>1$ corresponds to a maximal irreducible standard parabolic subgroup of $\langle\pi(S)\rangle$ generated by $s-1$ transpositions.
Let $\Orb(S)=[O_1,\ldots,O_r]$ be the orbits of $\langle\pi(S)\rangle$, arranged in their natural order (that is, such that for $1\le a<b\le r$, one has $i<j$ for any $i\in O_a$ and any $j\in O_b$), and let $\deg(S) = \big[|O_1|,\ldots,|O_r|\big]$ be their sizes.

The above immediately yields the following lemma.

\begin{lemma}\label{L:SubsetVsSequence}
For any sequence $A=[a_1,\ldots,a_r]$ with $\sum A=n$, there exists a unique subset $S\in\PP(\A)$ such that $\deg(S)=A$.
\end{lemma}
\begin{proof}
The sequence $\deg(S)$, together with the fact that the orbits are arranged in their natural order, determines the orbits of $\langle\pi(S)\rangle$.
Moreover, the points $i$ and $i+1$ are in the same orbit of $\langle\pi(S)\rangle$ if and only if $\sigma_i\in S$.
\end{proof}

Recall from \autoref{S:Counting} that for two sequences $A$ and $B$ of non-negative integers, each of whose sum is $n$, we denote by $N_{A,B}$ the number of simple vertex-labelled bipartite graphs $(U,V,E)$ for which $\deg_U=A$ and $\deg_V=B$.
Further, recall the automaton recognising normal forms described in \autoref{S:FinishingSets}; we call this automaton $\T$. The states of $\T$ are subsets of $\A$, namely the finishing set of the last read symbol $x_i\in D$.

\begin{theorem}\label{T:Transitions}%
If $X,Y\subseteq\A$ and $Y\ne\emptyset$, then the number of transitions from the state $X$ of $\T$ to an accept state $Y'$ of $\T$ with $Y\subseteq Y'$ is $N_{\deg(Y),\deg(\A\setminus X)}$.
\end{theorem}
\begin{proof}
A transition $s$ from the state $X$ ends in an accept state if and only if $S(s)\subseteq X$.  By \autoref{L:Transitions}, it only remains to show that%
\[
 N_{\deg(Y),\deg(\A\setminus X)} = \bigg|\left\{ t\in D\;\;\big|\;\;
                     t\big({\textstyle\rbigvee} Y\big)\wedge\big({\textstyle\bigvee}\A\setminus X\big)=1
    \text{\, and \;} t\rwedge\big({\textstyle\rbigvee} Y\big)=1 \right\}\bigg| \;.
\]
We consider vertex-labelled graphs with vertex set $\Orb(\A\setminus X)\cup \Orb(Y)$.
Let $\Orb(\A\setminus X)=[O_1,\ldots,O_r]$ and $\Orb(Y)=[P_1,\ldots,P_s]$.
Given $t\in D$ with $t\big({\textstyle\rbigvee} Y\big)\wedge\big({\textstyle\bigvee}\A\setminus X\big)=1$
and $t\rwedge\big({\textstyle\rbigvee} Y\big)=1$, we define a graph $\Gamma_t$ with vertex set $\Orb(\A\setminus X)\cup \Orb(Y)$ as follows:  For $i=1,\ldots,n$ we add an edge from $O_a$ to~$P_b$, where $O_a$ is the element in $\Orb(\A\setminus X)$ containing $i$ and $P_b$ is the element in $\Orb(Y)$ containing $\pi_t(i)$; it is obvious that this yields a bipartite graph and that the vertices have degrees $\deg(\A\setminus X)$ respectively $\deg(Y)$.

Let $u={\textstyle\rbigvee} Y$ and $s=tu$. As $t\wedge\big({\textstyle\bigvee}\A\setminus X\big)\preccurlyeq s\wedge\big({\textstyle\bigvee}\A\setminus X\big)=1$, we have $\pi_t(j)<\pi_t(j+1)$ if $\{j,j+1\}\subseteq O_a$ for some $a$.  Similarly, as $t\rwedge\big({\textstyle\rbigvee} Y\big)=1$, we have $\pi_t^{-1}(k)<\pi_t^{-1}(k+1)$ if $\{k,k+1\}\subseteq P_b$ for some $b$.
If there were $i,i'\in O_a$ with $i<i'$ and $\pi_t(i),\pi_t(i')\in P_b$ for some $a$ and $b$, the above would imply $\{i,i+1,\ldots,i'\}\subseteq O_a$ and $\{\pi_t(i),\pi_t(i+1),\ldots,\pi_t(i')\}\subseteq P_b$, with $\pi_t(i+1)=\pi_t(i)+1$.  In particular, $\sigma_{\pi_t(i)}\preccurlyeq u$, and then
$\pi_s(i+1)=\pi_u(\pi_t(i+1))=\pi_u(\pi_t(i)+1)<\pi_u(\pi_t(i))=\pi_s(i)$, contradicting $s\wedge\big({\textstyle\bigvee}\A\setminus X\big)=1$.  We have thus shown that the graph $\Gamma_t$ is simple.%
\medskip

Conversely, let $\Gamma$ be a simple vertex-labelled bipartite graph with $n$ edges and vertex classes $\Orb(\A\setminus X)$ and $\Orb(Y)$.  We construct $\pi_{t_\Gamma}\in\Sym_n$, and thus $t_{\Gamma}\in D$, by defining $\pi_{t_\Gamma}(i)$ for $i=1,\ldots,n$ (in this order):  If $i\in O_a$, then let $\pi_{t_\Gamma}(i)$ be the smallest number in the union of the vertices adjacent to $O_a$ that has not been chosen as an image.
If $\{j,j+1\}\subseteq O_a$ for some $a$ then by construction, $\pi_{t_\Gamma}(j)<\pi_{t_\Gamma}(j+1)$, implying $t_{\Gamma}\wedge\big({\textstyle\bigvee}\A\setminus X\big)=1$.
Similarly, if $\{k,k+1\}\subseteq P_b$ for some $b$ then by construction $\pi_{t_\Gamma}^{-1}(k)<\pi_{t_\Gamma}^{-1}(k+1)$,
implying $t\rwedge\big({\textstyle\rbigvee} Y\big)=1$.
Finally, let $u={\textstyle\rbigvee} Y$ and $s=t_{\Gamma}\,u$.
The fact that $\Gamma$ is simple implies that, if $\{j,j+1\}\subseteq O_a$ for some~$a$, then $\pi_{\Gamma_t}(j)$ and $\pi_{\Gamma_t}(j+1)>\pi_{\Gamma_t}(j)$ lie in different orbits of $\langle\pi(Y)\rangle$, whence $\pi_s(j+1)=\pi_u(\pi_{\Gamma_t}(j+1))>\pi_u(\pi_{\Gamma_t}(j))=\pi_s(j)$.
Thus, $s\wedge\big({\textstyle\bigvee}\A\setminus X\big)=1$.
\medskip

It is obvious from the construction that one has $\Gamma_{t_{\Gamma}}=\Gamma$ for any simple vertex-labelled bipartite graph $\Gamma$ with $n$ edges and vertex classes $\Orb(\A\setminus X)$ and $\Orb(Y)$.
Conversely, let $t_1,t_2\in D$ such that $t_i\big({\textstyle\rbigvee} Y\big)\wedge\big({\textstyle\bigvee}\A\setminus X\big)=1$ and $t_i\rwedge\big({\textstyle\rbigvee} Y\big)=1$ for $i=1,2$, and assume $\Gamma_{t_1}=\Gamma_{t_2}$.  The latter implies that $\pi_{t_1}=\pi_{t_2}$ up to permutations of the points within each orbit (or vertex).
More precisely, there exist $\alpha\in\langle\pi(\A\setminus X)\rangle$ and $\beta\in\langle\pi(Y)\rangle$ such that $\pi_{t_2}=\alpha \pi_{t_1}\beta$.
However, if $O_a=\{j,j+1,\ldots,j+k\}$ and $P_{b_0},\ldots,P_{b_k}$ are the vertices adjacent to $O_a$, where $b_0<\cdots<b_k$, then $t_i\wedge\big({\textstyle\bigvee}\A\setminus X\big)=1$ for $i=1,2$ implies that $\pi_{t_i}(j+\ell)\in P_{b_\ell}$ for $\ell=0,\ldots,k$ and $i=1,2$; that is, $\alpha=1$.
Similarly, $t_i\rwedge\big({\textstyle\rbigvee} Y\big)=1$ for $i=1,2$ implies $\beta=1$, that is, $\pi_{t_1}=\pi_{t_2}$ and thus $t_1=t_2$.
We have then shown that the maps $t\mapsto\Gamma_t$ and $\Gamma\mapsto t_{\Gamma}$ are inverses of each other, completing the proof.
\end{proof}

Let $T$ be the transition matrix of $\T$; for $X,Y\subseteq\A$ the entry $T_{Y,X}$ is the number of transitions from the state $X$ of $\T$ to the state $Y$ of $\T$.

\begin{corollary}\label{C:Transitions}%
If $X_1,X_2\subseteq\A$ such that the partitions of $n$ given by $\deg(\A\setminus X_1)$ respectively $\deg(\A\setminus X_2)$ are equal, then $T_{Z,X_1}=T_{Z,X_2}$ for all $Z\subseteq\A$, that is, the columns $X_1$ and~$X_2$ of the transition matrix of $\T$ are identical.
\end{corollary}
\begin{proof}
By \autoref{L:CntGraphs-Symm}, we have $N_{\deg(Y),\deg(\A\setminus X_1)}=N_{\deg(Y),\deg(\A\setminus X_2)}$ for all $Y\subseteq\A$.
Applying \autoref{T:Transitions} for all $\emptyset\ne Y\subseteq\A$ and using M\"obius inversion~\cite{Rota1964} for the poset $(\PP(\A),\supseteq)$, one obtains $T_{Z,X_1}=T_{Z,X_2}$ for all $Z\ne\emptyset$.
As the sum of each column of $T$ is $n!$, the claim for $Z=\emptyset$ follows.
\end{proof}

\begin{myremark}\label{R:Transitions}%
It is not true in general that $T_{Z_1,X}=T_{Z_2,X}$, if the partitions of~$n$ given by $\deg(Z_1)$ respectively $\deg(Z_2)$ are equal.
The reason is that the equality of partitions does not necessarily extend to all supersets of $Z_1$ respectively $Z_2$.

As an example, consider $Z_1=\{\sigma_1\}$ and $Z_2=\{\sigma_2\}$ for $B_4^+$ with $\A=\{\sigma_1,\sigma_2,\sigma_3\}$.
One has $\deg(Z_1)=[2,1,1]$ and $\deg(Z_2)=[1,2,1]$, so the partitions of $4$ that correspond to $\deg(Z_1)$ and $\deg(Z_2)$ are equal.
However, for $Z'=\{\sigma_1,\sigma_3\}\supseteq Z_1$ one has $\deg(Z')=[2,2]$, but no superset of $Z_2$ yields this partition of $4$: $\deg(\{\sigma_1,\sigma_2\})=[3,1]$, $\deg(\{\sigma_2,\sigma_3\})=[1,3]$, and $\deg(\A)=[4]$.

That is, the symmetry of $N_{\deg(Y),\deg(\A\setminus X)}$ under resorting of $\deg(Y)$ is broken by the M\"obius inversion.
\end{myremark}

\begin{myexample}\label{E:B4_partition}
Consider again the braid monoid $B_4^+$.
The following table lists, for each state, the orbits of the subgroup of $\Sym_4$ generated by the complement of the finishing set, and the partition of $4$ given by the sizes of these orbits.
\begin{center}
\begin{tabular}{ccccc}
 \textbf{state}
  & $F$
  & $\A\setminus F$
  & \textbf{$\mathbf{Orb}(\A\setminus F)$}
  & \textbf{partition of 4} \\ \hline
\rule[11pt]{0pt}{0pt}%
 $1$
  & $\emptyset$
  & $\A$
  & $\big[\{ 1,2,3,4 \}\big]$
  & $[4]$     \\[4pt]
 $2$
  & $\{ \sigma_1 \}$
  & $\{ \sigma_2, \sigma_3 \}$
  & $\big[\{ 1 \}\,; \{ 2,3,4 \}\big]$
  & $[3,1]$   \\[4pt]
 $3$
  & $\{ \sigma_2 \}$
  & $\{ \sigma_1, \sigma_3 \}$
  & $\big[\{ 1,2 \}\,; \{ 3,4 \}\big]$
  & $[2,2]$   \\[4pt]
 $4$
  & $\{ \sigma_1, \sigma_2 \}$
  & $\{ \sigma_3 \}$
  & $\big[\{ 1 \}\,; \{ 2 \}\,; \{ 3,4 \}\big]$
  & $[2,1,1]$ \\[4pt]
 $5$
  & $\{ \sigma_3 \}$
  & $\{ \sigma_1, \sigma_2 \}$
  & $\big[\{ 1,2,3 \}\,; \{ 4 \}\big]$
  & $[3,1]$   \\[4pt]
 $6$
  & $\{ \sigma_1, \sigma_3 \}$
  & $\{ \sigma_2 \}$
  & $\big[\{ 1 \}\,; \{ 2,3 \}\,; \{ 4 \}\big]$
  & $[2,1,1]$ \\[4pt]
 $7$
  & $\{ \sigma_2, \sigma_3 \}$
  & $\{ \sigma_1 \}$
  & $\big[\{ 1,2 \}\,; \{ 3 \}\,; \{ 4 \}\big]$
  & $[2,1,1]$ \\[4pt]
 $8$
  & $\A$
  & $\emptyset$
  & $\big[\{ 1 \}\,; \{ 2 \}\,; \{ 3 \}\,; \{ 4 \}\big]$
  & $[1,1,1,1]$
\end{tabular}
\end{center}
Observe that the columns of $T$ (cf.\ \autoref{E:B4}) corresponding to two states are identical, if (and only if) the partitions of~$4$ given by the complements of the finishing sets of the states are equal.  Note also that there are no three identical rows of $T$, although, up to reordering, $\deg(X_i)=[2,1,1]$ for each of the three sets $X_i=\{\sigma_i\}$ for $i=1,2,3$  (cf.\ \autoref{R:Transitions}).
\end{myexample}
\medskip

The idea now is to group states $X\in\PP(\A)$ into classes corresponding to the partition of $n$ given by $\deg(\A\setminus X)$, and to work with a matrix whose rows and columns are indexed by these classes.

\begin{notation}\label{N:Reduction}\quad\vspace*{-\itemsep}
\begin{enumerate}[label=\upshape{(\roman{*})},ref=\roman{*}]
\item\label{N:Reduction-partitions}
Let $\Part$ denote the set of partitions of $n$.
Let $\varphi:\PP(\A)\to\Part$ be the map assigning to each $X\in\PP(\A)$ the partition of $n$ given by $\deg(\A\setminus X)$, and choose $\psi:\Part\to\PP(\A)$ such that $\varphi(\psi(\alpha))=\alpha$ for all $\alpha\in\Part$.

\item\label{N:Reduction-transition}
Define $\TT\in\ZZ^{\Part\times\Part}$ by setting $\displaystyle\TT_{\beta,\alpha} = \sum_{Y\in\varphi^{-1}(\beta)} T_{Y,\psi(\alpha)}$.

By \autoref{C:Transitions}, $\TT_{\beta,\alpha}$ does not depend on the choice of $\psi$.

\item\label{N:Reduction-projection}
Define $P\in\ZZ^{\Part\times\PP(\A)}$ and $Q\in\ZZ^{\PP(\A)\times\Part}$ as follows:
\[
 P_{\alpha,X} =
   \begin{cases}
     1 & \text{if $\alpha=\varphi(X)$} \\
     0 & \text{otherwise}
   \end{cases}
\quad\text{respectively}\quad
 Q_{X,\alpha} =
   \begin{cases}
     1 & \text{if $X=\psi(\alpha)$} \\
     0 & \text{otherwise}
   \end{cases}
\;.
\]

\item\label{N:Reduction-uncorrected}
Define $\widetilde{N}\in\ZZ^{\Part\times\Part}$ as follows:
\[
 \widetilde{N}_{\alpha,\beta} =
   \begin{cases}
     n!               & \text{if $\alpha=[1,\ldots,1]$} \\
     N_{\alpha,\beta} & \text{otherwise}
   \end{cases} \;.
\]

\item\label{N:Reduction-complement/sum}
Define $C\in\ZZ^{\PP(\A)\times\PP(\A)}$ and $S\in\ZZ^{\PP(\A)\times\PP(\A)}$ as follows:
\[
 C_{X,Y} =
   \begin{cases}
     1 & \text{if $X=\A\setminus Y$} \\
     0 & \text{otherwise}
   \end{cases}
\quad\text{respectively}\quad
 S_{X,Y} =
   \begin{cases}
     1 & \text{if $X\subseteq Y$} \\
     0 & \text{otherwise}
   \end{cases} \;.
\]
\end{enumerate}
\end{notation}

\begin{proposition}\label{P:Reduction}\quad\vspace*{-\itemsep}
\begin{enumerate}[label=\upshape{(\roman{*})},ref=\roman{*}]
\item\label{P:Reduction-Moebius}
The matrix $S$ is invertible and $S^{-1}$ is the matrix given by
\[ (S^{-1})_{X,Y} =
   \begin{cases}
     (-1)^{|Y\setminus X|} & \text{if $X\subseteq Y$} \\
     0                     & \text{otherwise}
   \end{cases} \;.
\]

\item\label{P:Reduction-transition}
One has $T=S^{-1}CP^t\widetilde{N}P$ and $\widetilde{T}=PTQ=PS^{-1}CP^t\widetilde{N}$.

\item\label{P:Reduction-growthfunction}
Let $u\in\ZZ^{\PP(\A)\times1}$ and $v\in\ZZ^{1\times\PP(\A)}$ be defined by
\[
 u_{X} =
   \begin{cases}
     1 & \text{if $X = \A$} \\
     0 & \text{otherwise}
   \end{cases}
\quad\text{respectively}\quad
 v_{X} =
   \begin{cases}
     0 & \text{if $X = \emptyset$} \\
     1 & \text{otherwise}
   \end{cases}
\;,
\]
and let $\widetilde{u}\in\ZZ^{\Part\times1}$ and $\widetilde{v}\in\ZZ^{1\times\Part}$ be defined by
\[
 \widetilde{u}_{\alpha} =
   \begin{cases}
     1 & \text{if $\alpha = \varphi(\A) = [1,\ldots,1]$} \\
     0 & \text{otherwise}
   \end{cases}
\quad\text{and}\quad
 \widetilde{v}_{\alpha} =
   \begin{cases}
     0 & \text{if $\alpha = \varphi(\emptyset) = [n]$} \\
     1 & \text{otherwise}
   \end{cases}
\;.
\]
Then for any $k\in\NN$, one has $vT^ku=\widetilde{v}\widetilde{T}^k\widetilde{u}$.
In particular, the growth function of $B_n^+$ can be computed using the matrix $\widetilde{T}$.
\end{enumerate}
\end{proposition}
\begin{proof}
\begin{enumerate}[label=\upshape{(\roman{*})},ref=\roman{*}]
\item
The claim is just a restatement of M\"obius inversion~\cite{Rota1964} for the poset $(\PP(\A),\supseteq)$.

\item
It is elementary to check that $\widetilde{T}=PTQ$, that $PQ=1\in\ZZ^{\Part\times\Part}$, and that
\[
 (CP^t\widetilde{N}P)_{Y,X} =
  \begin{cases}
    n!                              & \text{if $Y=\emptyset$} \\
    N_{\deg(Y),\deg(\A\setminus X)} & \text{otherwise}
  \end{cases} \;.
\]
By \autoref{T:Transitions} and part (i), the latter implies $T=S^{-1}CP^t\widetilde{N}P$ proving the claim.

\item
Since $\varphi^{-1}\big([n]\big)=\{\emptyset\}$ and $\varphi^{-1}\big([1,\ldots,1]\big)=\{\A\}$, one has $\widetilde{v}P=v$ respectively $Q\widetilde{u}=u$.
Since
\[
  (QP)_{X,Y} =
    \begin{cases}
      1 & \text{if $X=\psi(\varphi(Y))$} \\
      0 & \text{otherwise}
    \end{cases}
\]
and $T_{X,\psi(\varphi(Y))}=T_{X,Y}$ for all $X,Y\in\PP(\A)$ by \autoref{C:Transitions}, one has $TQP=T$.
Thus, for any $k\in\NN$,
\[
  \widetilde{v}\widetilde{T}^k\widetilde{u}
   \;=\; \widetilde{v}(PTQ)^k\widetilde{u}
   \;=\; (\widetilde{v}P)T^k(Q\widetilde{u})
   \;=\; vT^ku \;.
\]\vskip-4ex
\end{enumerate}
\end{proof}

\begin{myexample}\label{E:B4_reduction}
Consider again the braid monoid $B_4^+$; recall \autoref{E:B4} and \autoref{E:B4_partition}.
With $\Part = \big\{ [4],  [3,1],  [2,2],  [2,1,1],  [1,1,1,1] \}$, one has
\[
 N=\begin{pmatrix*}[r]
      0 &   0 &   0 &  0 &  1 \\
      0 &   0 &   0 &  1 &  4 \\
      0 &   0 &   1 &  2 &  6 \\
      0 &   1 &   2 &  5 & 12 \\
      1 & \;4 & \;6 & 12 & 24
   \end{pmatrix*}
\qquad
 \widetilde{N}=
   \begin{pmatrix*}[r]
      0 &  0 &  0 &  0 &  1 \\
      0 &  0 &  0 &  1 &  4 \\
      0 &  0 &  1 &  2 &  6 \\
      0 &  1 &  2 &  5 & 12 \\
     24 & 24 & 24 & 24 & 24
   \end{pmatrix*}
\]
\smallskip
\[
  P=\begin{pmatrix*}[r]
       1 & 0 & 0 & 0 & 0 & 0 & 0 & 0 \\
       0 & 1 & 0 & 0 & 1 & 0 & 0 & 0 \\
       0 & 0 & 1 & 0 & 0 & 0 & 0 & 0 \\
       0 & 0 & 0 & 1 & 0 & 1 & 1 & 0 \\
       0 & 0 & 0 & 0 & 0 & 0 & 0 & 1
    \end{pmatrix*}
\qquad
 Q=\begin{pmatrix*}[r]
      1 & 0 & 0 & 0 & 0 \\
      0 & 1 & 0 & 0 & 0 \\
      0 & 0 & 1 & 0 & 0 \\
      0 & 0 & 0 & 1 & 0 \\
      0 & 0 & 0 & 0 & 0 \\
      0 & 0 & 0 & 0 & 0 \\
      0 & 0 & 0 & 0 & 0 \\
      0 & 0 & 0 & 0 & 1
   \end{pmatrix*}
\]
\smallskip
\[
 C=\begin{pmatrix*}[r]
      0 & 0 & 0 & 0 & 0 & 0 & 0 & 1 \\
      0 & 0 & 0 & 0 & 0 & 0 & 1 & 0 \\
      0 & 0 & 0 & 0 & 0 & 1 & 0 & 0 \\
      0 & 0 & 0 & 0 & 1 & 0 & 0 & 0 \\
      0 & 0 & 0 & 1 & 0 & 0 & 0 & 0 \\
      0 & 0 & 1 & 0 & 0 & 0 & 0 & 0 \\
      0 & 1 & 0 & 0 & 0 & 0 & 0 & 0 \\
      1 & 0 & 0 & 0 & 0 & 0 & 0 & 0
   \end{pmatrix*}
\qquad
 S=\begin{pmatrix*}[r]
      1 & 1 & 1 & 1 & 1 & 1 & 1 & 1 \\
      0 & 1 & 0 & 1 & 0 & 1 & 0 & 1 \\
      0 & 0 & 1 & 1 & 0 & 0 & 1 & 1 \\
      0 & 0 & 0 & 1 & 0 & 0 & 0 & 1 \\
      0 & 0 & 0 & 0 & 1 & 1 & 1 & 1 \\
      0 & 0 & 0 & 0 & 0 & 1 & 0 & 1 \\
      0 & 0 & 0 & 0 & 0 & 0 & 1 & 1 \\
      0 & 0 & 0 & 0 & 0 & 0 & 0 & 1
   \end{pmatrix*}
\]
\smallskip
\[
 S^{-1}=\begin{pmatrix*}[r]
           1 & -1 & -1 &  1 & -1 &  1 &  1 & -1 \\
           0 &  1 &  0 & -1 &  0 & -1 &  0 &  1 \\
           0 &  0 &  1 & -1 &  0 &  0 & -1 &  1 \\
           0 &  0 &  0 &  1 &  0 &  0 &  0 & -1 \\
           0 &  0 &  0 &  0 &  1 & -1 & -1 &  1 \\
           0 &  0 &  0 &  0 &  0 &  1 &  0 & -1 \\
           0 &  0 &  0 &  0 &  0 &  0 &  1 & -1 \\
           0 &  0 &  0 &  0 &  0 &  0 &  0 &  1
        \end{pmatrix*}
\]
\smallskip
\[
 \widetilde{T}=
   \begin{pmatrix*}[r]
      24 & 21 & 19 & 13 &  1 \\
       0 &  2 &  2 &  4 &  6 \\
       0 &  1 &  2 &  3 &  5 \\
       0 &  0 &  1 &  4 & 11 \\
       0 &  0 &  0 &  0 &  1
   \end{pmatrix*}
\]
\end{myexample}

\medskip

Our objective is to compute $\widetilde{T}$ directly, without summing explicitly over the elements of $\PP(\A)$.
In the light of \autoref{P:Reduction} (\ref*{P:Reduction-transition}), the remaining problem is to compute $PS^{-1}CP^t\in\ZZ^{\Part\times\Part}$ efficiently, that is, to perform M\"obius inversion using classes of states rather than individual states.

\begin{notation}\label{N:Refinements}%
We say that $X\in\PP(\A)$ is of \emph{partition type} $\tau(X)=\alpha$, if the partition of $n$ given by $\deg(X)$ is $\alpha$.  (That is, $\tau(X)=\varphi(\A\setminus X)$.)

For $\alpha\in\Part$ let $N_{\alpha}$ be the number of subsets $X\in\PP(\A)$ of partition type $\alpha$.
Moreover, let $M\in\ZZ^{\Part\times\Part}$, where $M_{\alpha,\beta}$ is the number of subsets $Y\subseteq X$ of any fixed subset $X$ of partition type $\alpha$, such that $\tau(Y)=\beta$.
(It is obvious that $M_{\alpha,\beta}$ does not depend on the choice of $X$.)
\end{notation}

\begin{myexample}\label{E:Refinements}%
For $n=4$, consider $\alpha=[2,2]$ and $\beta=[2,1,1]$.  One has $M_{\alpha,\beta}=2$, since $X=\{\sigma_1,\sigma_3\}$ is of partition type $\alpha$, and $X$ has exactly two subsets of partition type $\beta$, namely $\{\sigma_1\}$ and $\{\sigma_3\}$.
\end{myexample}

\begin{proposition}\label{P:Refinements}%
Let $\{\!\!\{ a_1\!\!*\!m_1,\ldots,a_r\!\!*\!m_r \}\!\!\}$ denote a multiset with pairwise distinct elements $a_1,\ldots,a_r$ whose multiplicities are $\mu(a_i)=m_i$ for $i=1,\ldots,r$.%
\vspace*{-\itemsep}
\begin{enumerate}[label=\upshape{(\roman{*})},ref=\roman{*}]
\item\label{P:Refinements-N}
If $\alpha\in\Part$ and $\alpha=\{\!\!\{ a_1\!\!*\!m_1,\ldots,a_r\!\!*\!m_r \}\!\!\}$ as a multiset, then
\[
   N_{\alpha}=\frac{r!}{m_1!\cdots m_r!} \;.
\]

\item\label{P:Refinements-M}
For any $m\in\NN$, let $\Part_m$ denote the set of partitions of $m$ and define $\mathfrak{Q}_{m} = \{\!\!\{ \gamma\!*\!N_{\gamma} \mid \gamma\in\Part_{m} \}\!\!\}$ with $N_{\gamma}$ as in {\upshape(\ref*{P:Refinements-N})}.
Moreover, for $m_1,\ldots,m_s\in\NN$, let
$
  \kappa_{m_1,\ldots,m_s}:\; \mathfrak{Q}_{m_1}\times\cdots\times\mathfrak{Q}_{m_s}
                           \to \mathfrak{Q}_{m_1+\cdots+m_s}
$
be the map induced by the joining of partitions.

If $\alpha=[a_1,\ldots,a_r]\in\Part$, then $M_{\alpha,\beta}$ is the multiplicity of $\beta$ in the multiset
$\kappa_{a_1,\ldots,a_r}(\mathfrak{Q}_{a_1}\times\cdots\times\mathfrak{Q}_{a_r})$.

\item\label{P:Refinements-Cost}
The matrix $M$ can be computed in time $O\big(d^{\sqrt{n}}\big)$ with $d<2198$.
\end{enumerate}
\end{proposition}
\begin{proof}
Claims (\ref*{P:Refinements-N}) and (\ref*{P:Refinements-M}) are obvious from \autoref{L:SubsetVsSequence}.
For claim (\ref*{P:Refinements-Cost}) consider $\alpha=[a_1,\ldots,a_r]$ and compute $M_{\alpha,\beta}$ for all $\beta\in\Part$ as follows.
Start with the multiset $\mathfrak{Q}=\{\!\!\{[\;]\}\!\!\}$, where $[\;]$ is the partition of 0; then for $i=1,\ldots,r$ compute $\mathfrak{Q}'=\kappa_{a,a_i}(\mathfrak{Q}\times\mathfrak{Q}_{a_i})$ (where $a=a_1+\cdots+a_{i-1}$) and replace $\mathfrak{Q}$ by $\mathfrak{Q}'$.
Computing $N_{\gamma}$ for all $\gamma\in\Part_{a_i}$ has a total cost of at most $O(|\Part|n^2(\ln n)^2)\subset O(|\Part|^2)$.
At any time, $\mathfrak{Q}$ and $\mathfrak{Q}_{a_i}$ each contain at most $|\Part|$ different elements, so there are at most $|\Part|^2$ joins to compute in each step, each at a cost of at most $O(n)$.
All multiplicities are bounded by $2^n$, so the product or the sum of two multiplicities can be computed at a cost of at most $O(n^2)$.
Thus, the cost of each step is at most $O\big(|\Part|^2(1+n+n^2)\big)=O\big(n^2|\Part|^2\big)$.
The number $r$ of steps is bounded by~$n$.
Since there are $|\Part|$ possibilities for~$\alpha$, the matrix $M$ can be computed in time $O\big(n^3|\Part|^3\big)$.
The claim then follows using the asymptotic expression~\cite{HardyRamanujan1918}
\[
p(n)\sim\frac1{4n\sqrt3} e^{\pi\sqrt{\frac{2n}{3}}} \;.
\]\vskip-3.5ex
\end{proof}

\begin{proposition}\label{P:Moebius}
For $\alpha=[a_1,\ldots,a_r]\in\Part$, let $\sgn(\alpha) = (-1)^{n-r}$.
Then
\[
   \big(PS^{-1}CP^t\big)_{\alpha,\beta} \;=\;
     \sgn(\alpha)\,\sgn(\beta)\sum_{\gamma\in\Part}
        \sgn(\gamma) N_{\gamma} M_{\gamma,\alpha} M_{\gamma,\beta}
\]
for $\alpha,\beta\in\Part$.
In particular, $\big(PS^{-1}CP^t\big)_{\alpha,\beta} = \big(PS^{-1}CP^t\big)_{\beta,\alpha}$.
\end{proposition}
\begin{proof}
Recall that for $X\in\PP(\A)$, the orbits of $\langle \pi(X)\rangle$ correspond to the maximal irreducible standard parabolic subgroups of $\langle \pi(X)\rangle$.
More precisely, an orbit of size $s>1$ is generated by the image of a maximal set of consecutive atoms in $X$ that has size $s-1$.
Hence, if $\deg(X) = [a_1,\ldots,a_r]$, then $|X| = \sum_{i=1}^r (a_i-1) = n-r$.
In particular, $\sgn(\tau(X)) = (-1)^{|X|} = (-1)^{|\A|}(-1)^{|\A\setminus X|} = (-1)^{|\A|}\sgn(\varphi(X))$.

Let $\alpha,\beta\in\Part$ be fixed.
For $G\in\PP(\A)$ define
\[
  n_G = \Big|\big\{ (X,Y)\in\PP(\A)\times\PP(\A) \mid
             \tau(X)=\alpha,\; \tau(Y)=\beta, X\cup Y = G \big\}\Big|
\]
and
\[
  \widetilde{n}_G = \Big|\big\{ (X,Y)\in\PP(\A)\times\PP(\A) \mid
             \tau(X)=\alpha,\; \tau(Y)=\beta, X\cup Y\subseteq G \big\}\Big| \;.
\]
If $\tau(G)=\gamma$ then $\widetilde{n}_G = M_{\gamma,\alpha} M_{\gamma,\beta}$.
Moreover, by M\"obius inversion~\cite{Rota1964} for the poset $(\PP(\A),\subseteq)$, one has
\[
n_G = \sum_{\substack{G'\in\PP(\A) \\[0.5ex] G'\subseteq G}}
                  (-1)^{|G\setminus G'|}\; \widetilde{n}_{G'} \;.
\]
With \autoref{P:Reduction} (\ref*{P:Reduction-Moebius}) and the above, we obtain
\begin{align*}
\big(PS^{-1}CP^t\big)_{\alpha,\beta}
 &= \sum_{\substack{X\in\varphi^{-1}(\alpha) \\[0.5ex]
                       Y\in\varphi^{-1}(\beta) \\[0.5ex]
                       X\subseteq \A\setminus Y}}
          (-1)^{|\A\setminus Y|-|X|}
 \;\;= \sum_{\substack{X'\in\tau^{-1}(\alpha) \\[0.5ex]
                       Y'\in\tau^{-1}(\beta) \\[0.5ex]
                       X'\cup Y'=\A}}
          (-1)^{|Y'|-|\A\setminus X'|}
\end{align*}
\begin{align*}
 &= (-1)^n\,\sgn(\alpha)\,\sgn(\beta)\,n_{\A}
\\[2ex]
 &= (-1)^n\,\sgn(\alpha)\,\sgn(\beta)\,\sum_{G'\in\PP(\A)}
                                (-1)^{|\A\setminus G'|}\; \widetilde{n}_{G'}
\\[2ex]
 &= \sgn(\alpha)\,\sgn(\beta)\,\sum_{\gamma\in\Part}
                     N_{\gamma}\,\sgn(\gamma)\; M_{\gamma,\alpha} M_{\gamma,\beta}
\end{align*}
as claimed.
\end{proof}

\begin{corollary}\label{C:Moebius}
The matrix $PS^{-1}CP^t$ can be computed in time $O\big(d^{\sqrt{n}}\big)$ with $d<2198$.
\end{corollary}
\begin{proof}
The claim follows from \autoref{P:Moebius} and \autoref{P:Refinements} (\ref*{P:Refinements-Cost}), noting that all entries of $M$ are bounded by $2^n$, and that $n^2|\Part|^3\in O\big(d^{\sqrt{n}}\big)$.
\end{proof}

\begin{myexample}\label{E:B4_reduced}
Consider the braid monoid $B_4^+$; cf.\ \autoref{E:B4}, \autoref{E:B4_partition} and \autoref{E:B4_reduction}.
With $\Part = \big\{ [4],  [3,1],  [2,2],  [2,1,1],  [1,1,1,1] \}$, one has
\[
 M=\begin{pmatrix*}[r]
      1 & 2 & 1 & 3 & 1 \\
      0 & 1 & 0 & 2 & 1 \\
      0 & 0 & 1 & 2 & 1 \\
      0 & 0 & 0 & 1 & 1 \\
      0 & 0 & 0 & 0 & 1
   \end{pmatrix*}
\qquad
 PS^{-1}CP^t=
   \begin{pmatrix*}[r]
      -1 &  2 &  1 & -3 & \;\;1 \\
       2 & -2 & -2 &  2 &  0 \\
       1 & -2 &  0 &  1 &  0 \\
      -3 &  2 &  1 &  0 &  0 \\
       1 &  0 &  0 &  0 &  0
   \end{pmatrix*}
\]
\end{myexample}


\section{Feasibility in practice}\label{S:Feasibility}%
While its time complexity and its space complexity are still exponential in $n$, the new method presented here allows to compute the growth function of braid monoids for much larger values of $n$ than the methods from~\cite{Charney1993}.

\autoref{T:Timing} shows the time requirements for both methods.
Computations were done with a development version of \textsc{Magma} \cite{Magma} V2.18 on a GNU\,/\,Linux system with an Intel E8400 64-bit CPU (core:~3\,GHz, FSB:~1333\,MHz) and a main memory bandwidth of 4.7\,GB/s (X38 chipset, dual channel DDR2 RAM, memory bus: 1066\,MHz).
The maximum amount of memory used was 1370\,MB.
\begin{table}[ht]
\begin{center}
\begin{tabular}{@{\,}l|c@{\hspace{9pt}}c@{\hspace{9pt}}c@{\hspace{9pt}}c@{\hspace{9pt}}c@{\hspace{9pt}}c@{\hspace{9pt}}c@{\hspace{9pt}}c@{\hspace{9pt}}c@{\hspace{9pt}}c@{\hspace{1.5pt}}}
$n$ &
  7 &
  8 &
  9 &
 10 &
 15 &
 20 &
 25 &
 28 &
 29 &
 30
\\ \hline
$t${\footnotesize$_{\text{\cite{Charney1993}}}$}\rule[12pt]{0pt}{0pt} &
  0.6s &
  8.5s &
  145s &
 2820s &
 -- &
 -- &
 -- &
 -- &
 -- &
 --
\\[0.5ex]
$t${\tiny$_{\text{[G]}}$} &
 \textless\,0.1s &
 \textless\,0.1s &
 \textless\,0.1s &
 \textless\,0.1s &
 \textless\,0.1s &
  0.53s &
  17.1s &
  95.8s &
  191s &
  359s
\end{tabular}
\end{center}
\caption{Time $t${\footnotesize$_{\text{\cite{Charney1993}}}$} required for the method from \cite{Charney1993}, and time $t${\footnotesize$_{\text{[G]}}$} required for the new method.}\label{T:Timing}%
\end{table}

Computing the transition matrix of the more efficient of the automata from~\cite{Charney1993} (cf.~\autoref{S:FinishingSets}) involves considering $n!$ transitions for each of the $2^{n-1}$ states, so has time complexity $O(n!\,2^n)$; the space required is $O(4^n)$.
On our computer, the computation time would well exceed 10 hours for $n=11$; the transition matrix would exceed 4\,GB for $n=16$ and 2\,EB ($2\cdot10^{18}$\,B) for $n=30$.

The new method described in this paper, on the other hand, has time complexity $O\big(d^{\sqrt{n}}\big)\subset O(2^n)$ and requires space $O\big(|\Part|^2\big)=O\big(n^{-2}f^{\sqrt{n}}\big)\subset O(2^n)$ with $f<170$.
While the methods from~\cite{Charney1993} become infeasible if $n\gtrsim10$, the new method is usable at least up to $n\approx 30$ (cf.~\autoref{T:Timing}), removing a significant obstacle for the development of practical random generators.


\section*{Acknowledgements}%
The author wishes to thank Dr Stephen Tawn for porting an implementation of the described algorithm as a \Magma\ package to the \Magma\ C-kernel.  The timings reported in \autoref{S:Feasibility} were obtained using his implementation.

Further thanks go the anonymous referees for their thorough reviews of the article and helpful comments.


\bigskip

\noindent\textbf{Volker Gebhardt}\\
\noindent School of Computing, Engineering and Mathematics\\
University of Western Sydney\\
Locked Bag 1797, Penrith NSW 2751, Australia\\
\noindent E-mail: \texttt{v.gebhardt@uws.edu.au}

\end{document}